\documentclass[12pt]{amsart}
 \newtheorem{theo}{Theorem}[section]
 \newtheorem{coro}[theo]{Corollary}
 \newtheorem{lemma}[theo]{Lemma}
 \newtheorem{prop}[theo]{Proposition}
 \theoremstyle{definition}
 
 \theoremstyle{remark}
 \newtheorem{remark}[theo]{Remark}

\newtheorem{example}[theo]{Example}

\newtheorem{problem}[theo]{Problem}
\newtheorem*{question*}{Question}
 
 \numberwithin{equation}{section}

\usepackage{amsmath}
\usepackage{amsfonts}
\usepackage{amssymb}
\usepackage{mathrsfs}
\usepackage{stmaryrd}

\newcommand{\ZZ}{{\mathbb Z}}
\newcommand{\NN}{{\mathbb N}}
\newcommand{\RR}{{\mathbb R}}

\newcommand{\Ss}{{\mathscr S}}
\newcommand{\TT}{{\mathscr T}}

\newcommand{\xv}{{\mathbf x}}

\newcommand{\av}{{\mathbf a}}

\newcommand{\ein}{{\mathbf 1}}

\newcommand{\implyr}{\Longrightarrow}
\newcommand{\implylr}{\Longleftrightarrow}

\renewcommand{\phi}{\varphi}

\DeclareMathOperator{\sos}{SOS}

\begin{document}

\title{Sums of integers and sums of their squares}

\author{Detlev W.~Hoffmann}
\address{Fakult\"at f\"ur Mathematik,
Technische Universit\"at Dortmund,
D-44221 Dortmund,
Germany}
\email{detlev.hoffmann@math.tu-dortmund.de}

\thanks{The author is supported in part by DFG grant HO 4784/2-1
{\em Quadratic forms,
quadrics, sums of squares and Kato's cohomology in positive characteristic}.}

\date{}

\begin{abstract}
Suppose a positive integer $n$ 
is written as a sum of squares of
$m$ integers.  What can one say about the value $T$ of the sum of these $m$ integers
itself?  Which $T$ can be obtained if one considers all possible
representations of $n$ as a sum of squares of $m$ integers?  Denoting this
set of all possible $T$ by $\Ss_m(n)$,  
Goldmakher and Pollack have given a simple characterization of 
$\Ss_4(n)$ using elementary arguments.  Their result can be reinterpreted
in terms of Mordell's theory of representations of binary integral 
quadratic forms
as sums of squares of integral linear forms.   
Based on this approach, we characterize 
$\Ss_m(n)$ for all $m\leq 11$ and provide a few partial results for arbitrary $m$.
We also show how Mordell's results can be used to study variations of the original
problem where the sum of the integers is replaced by a linear
form in these integers.  In this way, we recover and generalize earlier 
results by Z.W.~Sun et.~al..

\end{abstract}


\maketitle

\section{Introduction}
Let $\NN$ (resp. $\NN_0$) denote the set of all positive
(resp. nonnegative) integers.  For $m\in\NN$, we define
 $$\sos(m)=\{ a_1^2+a_2^2+\ldots +a_m^2\,|\, 
\mbox{$a_i\in\ZZ$ for $1\leq i\leq m$}\}$$
to be the set of all nonnegative integers that are sums of 
$m$ integer squares.
By Lagrange's 4-square theorem, we know that for $m\geq 4$, we have 
$\sos(m)=\NN_0$.
Now let $n\in \NN$, $T\in\ZZ$, and consider the following
system of diophantine equations:
\begin{equation}\label{system}
\begin{array}{rcrcccrcl}
x_1 & + & x_2 & + & \ldots & + & x_m & = & T\\[1ex] 
x_1^2 & + & x_2^2 & + & \ldots & + & x_m^2 & = & n
\end{array}
\end{equation}
To have a solution, some necessary conditions must be satisfied
by $n$ and $T$:
\begin{equation}\label{cond}
\begin{array}{rl}
\mbox{(1)} &  n\in \sos(m);\\
\mbox{(2)} &  n\equiv T\bmod 2;\\
\mbox{(3)} &  T^2\leq mn.
\end{array}
\end{equation}
Condition (1) is obvious and poses no restriction on $n\in\NN$
if $m\geq 4$.  Condition (2) follows from the fact that for each
$x\in\ZZ$, one has $x^2\equiv x\bmod 2$.  Condition (3) follows from
the Cauchy-Schwarz inequality.  Indeed, if 
$\xv =(x_1,x_2,\ldots,x_m)\in \ZZ^m$ is a solution of the above system,
and putting $\ein =(1,1,\ldots, 1)\in \ZZ^m$, then the Cauchy-Schwarz
inequality applied to the usual scalar product (dot product) implies
\begin{equation}\label{cs}
T^2=(\ein\cdot\xv)^2\leq (\ein\cdot\ein)(\xv\cdot\xv)=mn\,.
\end{equation}
We now fix $n\in \sos(m)$ and define the set of those $T$ for which
the above system has a solution:
$$\Ss_m(n)=\{ T\in\ZZ\,|\,\mbox{Eq.\ \ref{system} has a solution $\xv\in\ZZ^m$}\}\,.$$
Note that $\Ss_m(n)$ is symmetric in the sense that $T\in \Ss_m(n)$ if
and only if $-T\in\Ss_m(n)$.
Goldmakher and Pollack \cite[Th.\ 2]{gp} have determined $\Ss_4(n)$:
\begin{theo} 
$\Ss_4(n)=\{ T\in\ZZ\,|\,T\equiv n\bmod 2,\ 4n-T^2\in \sos(3)\}$.
\end{theo}
In particular,  Legendre's 3-square theorem readily implies 
that for odd $n$, one always has
$1\in \Ss_4(n)$, thus giving a new proof of a conjecture of 
Euler\footnote{In their paper \cite{gp},
Goldmakher and Pollack mention Franz 
Lemmermeyer's earlier proof of Euler's conjecture on
mathoverflow.net/questions/37278/euler-and-the-four-squares-theorem
that also makes use of Legendre's 3-square theorem.}.

The purpose of the present paper is to study the sets $\Ss_m(n)$
in more detail also for other values of $m$.  In particular, 
we get complete descriptions of $\Ss_m(n)$ for $m\leq 11$: see
Proposition \ref{easy} for the case $m=1$ and the case $n\leq m$,
Theorem \ref{m<8} for the case $2\leq m\leq 7$, and
Theorem \ref{m=8-11} for the case $8\leq m\leq 11$.
Some further results such as the determination of $\Ss_m(n)$ in the case
$10\leq m<n\leq m+6$ (Corollary~\ref{m>9}) are also included.

For small values of $m$ we use classic results
by Mordell on representations of integral binary forms as sums of 
squares of integral linear forms.  This approach also allows a 
new interpretation of Goldmakher and Pollack's results on 
$\Ss_4(n)$, and it can be applied to variations of the above problem
studied by Z.-W.\ Sun et.\ al.\ in a series of papers
\cite{sun1}, \cite{sun2}, \cite{sun3}, \cite{sun4}.
There, one considers modified systems of equations for $m=4$
where the first equation in  Eq.~\ref{system} is replaced
by some other integral polynomial equation, i.e., one asks for
solutions of 
$$\begin{array}{rcl}
x_1^2+x_2^2+x_3^2+x_4^2 & = & n\\[1ex]
P(x_1,x_2,x_3,x_4) & = & T
\end{array}$$
where $P(x_1,x_2,x_3,x_4)\in\ZZ[x_1,x_2,x_3,x_4]$.
In analogy to the above notation and for given $n\in\NN$,
we denote the set of those $T\in\ZZ$ for which this new system
has a solution $\xv\in \ZZ^4$ by  $\Ss_{4,P}(n)$.
Using Mordell's results, we show how several of the results
by  Z.-W.\ Sun et.\ al.\ concerning $\Ss_{4,P}(n)$ for 
linear polynomials $P$ can be easily
recovered and extended.

\section{General results}
We start with some easy observations.
\begin{prop}\label{easy} Let $n,m\in\NN$.
\begin{enumerate}
\item[(1)] If $n=a^2$ for some $a\in\NN$, then 
$\Ss_1(n)=\{ \pm a\}$, otherwise $\Ss_1(n)=\emptyset$. 
\item[(2)] $\Ss_m(n)\subseteq \Ss_{m+1}(n)$.
\item[(3)] If $n\leq m$, then 
$\Ss_m(n)=\{ T\in\ZZ\,|\,n\equiv T\bmod{2},\ |T|\leq n\}$.
\end{enumerate}
\end{prop}
\begin{proof} (1) is trivial.

(2) Any representation of $n$ by a a sum of $m$ squares
(with corresponding sum $T\in \Ss_m(n)$) becomes a representation 
by $m+1$ squares by adding $0^2$, hence $T=T+0\in \Ss_{m+1}(n)$.

(3) It is obvious that if $n\leq m$, then $\max\bigl(\Ss_m(n)\bigr)=n$ is
obtained by the representation $n=\sum_{i=1}^n 1^2$. So any $T\in \Ss_m(n)$
satisfies $|T|\leq n$. By changing the signs
of the $1$'s in the above sum as necessary, one gets that each $T\in\ZZ$ with 
$|T|\leq n$ and $n\equiv T\bmod{2}$ will be in $\Ss_m(n)$.
\end{proof}

In view of Eq.~\ref{cond}, it is a priori possible that
$T\in \Ss_m(n)$ with $T^2=mn$.  Of course, since $T\in\ZZ$, 
for this to hold, $mn$ must be a square.  More precisely, we have
the following.
\begin{prop}\label{max} Let $m,n\in\NN$ and $T\in\ZZ$. Then
$T\in \Ss_m(n)$ with $T^2=mn$ if and only if there exists
$a\in\NN$ with $n=ma^2$, in which case $T=\pm ma$.\end{prop}
\begin{proof}
If $n=ma^2$, then $n=\sum_{i=1}^m a^2$ and 
$T=\sum_{i=1}^m a=ma\in \Ss_m(n)$.
Conversely, if $T\in\Ss_m(n)$ with $T^2=mn$, then there exists
$\xv\in\ZZ^m$ satisfying equality in the Cauchy-Schwarz inequality
Eq.~\ref{cs}, which implies that $\xv$ and the vector $\ein$
are linearly dependent, from which it follows that 
$\xv=(x_1,x_2,\ldots,x_m)=(a,a,\ldots,a)$
for some $a\in \ZZ$, and plugging this into Eq.~\ref{system}
implies that $n=ma^2$ and $T^2=mn=m^2a^2$.
\end{proof}
Because of this result, it makes sense to focus mainly on
those $T$ in $\Ss_m(n)$ with $T^2<mn$. 
For $n,m\in\NN$, we define the following sets of all $T\in\ZZ$ satisfying
condition (2) and strict inequality in (3) in Eq.~\ref{cond}:
$$\begin{array}{rcl}
\TT_m(n) & = & \{ T\in\ZZ\,|\,n\equiv T\bmod{2},\ T^2< mn\}\,,\\
\Ss_m'(n) & = & \{ T\in \Ss_m(n)\,|\,T^2<mn\}\,.
\end{array}$$
We say that $\Ss_m(n)$ is \emph{full} if $\Ss_m'(n)=\TT_m(n)$.
As an immediate consequence of Proposition \ref{max},
Eq.~\ref{cond} and the definition of fullness, we get the following.
\begin{coro}  Let $m,n\in\NN$.
$$\mbox{$\Ss_m(n)$ is full}\quad\implylr
\quad\left\{ \begin{array}{ll}
\Ss_m(n)=\TT_m(n)\cup\{ \pm ma\} & \mbox{if\ $\exists\,a\in\NN:\,n=ma^2$;}\\
\Ss_m(n)=\TT_m(n) & \mbox{otherwise.}
\end{array}\right.$$
\end{coro}

\begin{coro}\label{small-n}
Let $n,m\in\NN$ with $n\leq m$.  Then
$\Ss_m(n)$ is full if and only if $m \leq n+4+\frac{4}{n}$.
In particular, $\Ss_m(n)$ is full whenever $n\leq m\leq 8$. 
\end{coro}
\begin{proof}  It follows readily from Prop.~\ref{easy}(2)
and the definition of fullness
that $\Ss_m(n)$ is full iff $|T|\leq n$ for all 
$T\in\ZZ$ with $n\equiv T\bmod{2}$
and $T^2< mn$,  iff $(n+2)^2\geq mn$ iff 
$m \leq n+4+\frac{4}{n}$.  Since $n+4+\frac{4}{n}\geq 8$ for all
$n\in\NN$, it follows that $\Ss_m(n)$ is full whenever $n\leq m\leq 8$. 
\end{proof}

\section{Mordell's results on sums of squares of linear forms}
Let us fix $n,m\in\ZZ$. Working in $\ZZ[X,Y]$,
the polynomial ring in two variables over the integers,
one readily finds that having a solution $\xv=(x_1,\ldots,x_m)\in\ZZ^m$
of Eq.~\ref{system} is equivalent to having an equation
\begin{equation}\label{linear}
\sum_{i=1}^m(X+x_iY)^2=mX^2+2TXY+nY^2
\end{equation} 
in $\ZZ[X,Y]$. 
We will denote the binary integral quadratic form
$mX^2+2TXY+nY^2$ by $[m,T,n]$ for short.
Hence, the form
$[m,T,n]$ can be written as a sum of $m$ squares of 
integral linear forms.  Note that this in turn means that
the form $[m,T,n]$ must be positive semi-definite.
We define the determinant $\Delta$ of $[m,T,n]$ by
$$\Delta (m,T,n)=mn-T^2\,,$$
the determinant of the Gram matrix
$\begin{pmatrix} m & T\\ T & n\end{pmatrix}$ of the associated
bilinear form. 
By the Hurwitz criterion (and since $m,n\in\NN$), the form
$[m,T,n]$ is positive semi-definite 
iff $\Delta (m,T,n)=mn-T^2\geq 0$, 
i.e., $mn\geq T^2$, which is a different way of deriving
the necessary condition (3) in Eq.~\ref{cond}.

In \cite{mo1}, \cite{mo2}, Mordell considered the following 
more general problem.
Given $a,h,b\in\ZZ$ and $m\in\NN$, what are necessary and sufficient 
conditions so that $[a,h,b]$ can be written as a sum of 
$m$ squares of integral linear forms, i.e. so that there are
$a_i,b_i\in \ZZ$, $1\leq i\leq m$, with
\begin{equation}\label{eq-mordell1}
\sum_{i=1}^m(a_iX+b_iY)^2=aX^2+2hXY+bY^2\,?
\end{equation} 

The necessary and sufficient criteria found by Mordell
for the solvability of Eq.~\ref{eq-mordell1}, when applied to
$[a,h,b]=[m,T,n]$, thus become necessary conditions for
the solvability of Eq.~\ref{linear}.
For $2\leq m\leq 7$ 
we can say more (recall that the trivial case $m=1$ has been dealt with
in Prop.~\ref{easy}).
\begin{prop}\label{equiv}
Let $m,n\in\NN$ and $T\in\ZZ$.
\begin{enumerate}
\item[(i)]  If $m=2,3$, then $T\in\Ss_m(n)$ if and only if
$[m,T,n]$ is a sum of $m$ squares of integral linear forms.
\item[(ii)] If $4\leq m\leq 7$, then $T\in\Ss_m(n)$ if and only if
$T\equiv n\bmod 2$ and
$[m,T,n]$ is a sum of $m$ squares of integral linear forms.
\end{enumerate}
\end{prop}
\begin{proof}
If $T\in\Ss_m(n)$, then as remarked above,  $[m,T,n]$ is a sum of
$m$ squares of integral linear forms.  Note also that 
$T\equiv n\bmod 2$.

Conversely, suppose that $2\leq m\leq 8$, and 
in addition that $n\equiv T\bmod 2$
in the case $m\geq 4$, and assume that 
\begin{equation}\label{eq-mordell2}
[m,T,n]=\sum_{i=1}^m(a_iX+b_iY)^2\end{equation}
with $a_i,b_i\in\ZZ$.
We may assume that $a_i\geq 0$ for $1\leq i\leq m$
after changing signs
of both $a_i$ and $b_i$ whenever necessary.
Since $m=\sum_{i=1}^ma_i^2$, we see that for $m=2,3$ the only solution
is $a_1=\ldots=a_m=1$, which means we have a solution of 
Eq.~\ref{linear}, hence $T\in \Ss_m(n)$.

Now if $m=4$, our solution of Eq.~\ref{eq-mordell2} 
satisfies in particular $m=\sum_{i=1}^ma_i^2$, 
and after permuting the summands if necessary, 
there are two possibilities:
 $(a_1,\ldots ,a_4)=(1,1,1,1)$ or
$(a_1,\ldots ,a_4)=(2,0,0,0)$.
In the case $(1,1,1,1)$, this
again implies that we get a solution of 
Eq.~\ref{linear}, hence $T\in \Ss_4(n)$.

In the case
$(2,0,0,0)$, we get $4b_1=2T$ which 
necessarily implies that we have $n\equiv T\equiv 0\bmod 2$.
But then $n=\sum_{i=1}^4b_i^2\equiv \sum_{i=1}^4\pm b_i\equiv 0\bmod 2$.
We put
$$\begin{array}{rclcrcl}
c_1 & = & \frac{1}{2}(b_1+b_2+b_3+b_4) & , & 
c_2 & = & \frac{1}{2}(b_1+b_2-b_3-b_4) \\[1ex]
c_3 & = & \frac{1}{2}(b_1-b_2+b_3-b_4) & , & 
c_4 & = & \frac{1}{2}(b_1-b_2-b_3+b_4)
\end{array}$$
and we have $c_i\in\ZZ$, $\sum_{i=1}^4c_i=2b_1=T$ and 
$\sum_{i=1}^4c_i^2=\sum_{i=1}^4 b_i^2=n$ and therefore
$T\in\Ss_4(n)$.

If $5\leq m\leq 7$, our solution of Eq.~\ref{eq-mordell2} 
satisfies in particular $m=\sum_{i=1}^ma_i^2$, 
and after permuting the summands if necessary, 
there are two possibilities:
$(a_1,\ldots,a_m)=(1,1,\ldots,1)$ or
$(a_1,\ldots,a_m)=(2,0,0,0,1,\ldots,1)$.

In the case $(1,\ldots,1)$, we conclude as before that $T\in \Ss_m(n)$.
Now suppose that we are in the case $(2,0,0,0,1,\ldots,1)$.  Then
$T=2b_1+\sum_{i=5}^mb_i$.  Now
$$\sum_{i=5}^mb_i\equiv T\equiv n\equiv\sum_{i=1}^mb_i^2 \equiv
\sum_{i=1}^m\pm b_i\bmod 2$$
which implies that $\sum_{i=1}^4\pm b_i$ is even.
Hence, with the same $c_1,\ldots,c_4$ as above and with $c_i=b_i$
for $5\leq i\leq m$,
we have $c_i\in\ZZ$, $\sum_{i=1}^mc_i=2b_1+\sum_{i=5}^mb_i=T$ and 
$\sum_{i=1}^mc_i^2=\sum_{i=1}^m b_i^2=n$ and therefore
$T\in\Ss_m(n)$.
\end{proof}

We now turn to Mordell's results on solving Eq.~\ref{eq-mordell1} for
given $aX^2+2hXY+bY^2\in\ZZ[X,Y]$. Let us first deal with some
obvious cases. Recall that the solvability requires $a,b,\Delta=ab-h^2\geq 0$,
which we henceforth assume. Also, if we assume in addition $aX^2+2hXY+bY^2\neq 0$,
then we cannot have $a=b=0$, in which case we may assume $a>0$.

\begin{prop}\label{obvious}  
Let $aX^2+2hY^2+bY^2\in\ZZ[X,Y]\setminus\{ 0\}$ with $a>0$, $b,\Delta\geq 0$.
\begin{enumerate}
\item[(a)] If $\Delta=0$, then the following are equivalent:
\begin{enumerate}
\item[(i)] Eq.~\ref{eq-mordell1} is solvable;
\item[(ii)] $a\in\sos(m)$;
\item[(iii)] There exist $r,s,t\in\ZZ$ with 
$0<t\in\sos (m)$ and $aX^2+2hXY+bY^2=t(rX+sY)^2$.
\end{enumerate}
\item[(b)] Eq.~\ref{eq-mordell1} is solvable for $m=1$ iff
$\Delta=0$ and $a\in\sos(1)$.
\end{enumerate}
\end{prop}
\begin{proof}
In (a), ihe implications (i)$\implyr$(ii) and (iii)$\implyr$(i) are trivial.
If (ii) holds, then $ab=h^2$ implies that we can find $r,t\in\NN$, $s\in\NN_0$
with $a=r^2t$, $b=s^2t$ and $h=\pm rst$, and thus
$aX^2+2hXY+bY^2=t(rX\pm sY)^2$.  But then it is well known (or easy to check) that
$0<a\in\sos(m)$ iff $0<t\in\sos(m)$, which yields (iii).

Now if Eq.~\ref{eq-mordell1} is solvable for $m=1$, then
$aX^2+2hXY+bY^2=(rX+sY)^2$ for some $r,s\in\ZZ$.  Hence 
$\Delta=ab-h^2=r^2s^2-(rs)^2=0$.  (b) follows now readily from (a).
\end{proof}

Before we state Mordell's results, we introduce some further notations.
For $n\in\ZZ\setminus\{ 0\}$ and any prime number $p$, we denote by
$v_p(n)\in\NN_0$ the usual $p$-adic value of $n$, and by $n_p$ the
$p$-free part of $n$, so that
$n=p^{v_p(n)}n_p$ where $n_p\in\ZZ$ with $\gcd(p,n_p)=1$.
If $p$ does not divide $n$, then $\bigl(\frac{n}{p}\bigr)$ denotes
the usual Legendre symbol:
$$\left(\frac{n}{p}\right)=\left\{\begin{array}{rl}
1 & \mbox{if $n$ is a quadratic residue modulo $p$,}\\[1ex]
-1 & \mbox{otherwise.}
\end{array}\right.$$

Let now $aX^2+2hXY+bY^2\in\ZZ[X,Y]$ with $a,b,\Delta=ab-h^2>0$.
Let $d=\gcd (a,h,b)$ and $\widetilde{d}=\gcd (a,2h,b)$.  Then
$\widetilde{d}\in\{ d,2d\}$. 

\begin{theo}[Mordell \cite{mo1}, \cite{mo2}]\label{thm-mordell}
Let $aX^2+2hXY+bY^2\in\ZZ[X,Y]$ with $a,b,\Delta=ab-h^2>0$. 
\begin{enumerate}
\item[(i)]  Eq.~\ref{eq-mordell1} is solvable for $m=2$ iff
$\Delta\in\sos (1)$ and $d\in\sos (2)$.
\item[(ii)] Eq.~\ref{eq-mordell1} is solvable for $m=3$ iff all
of the following 
conditions are satisfied:
\begin{itemize}
\item[($\alpha$)] If $v_2(d)$ is odd then $\widetilde{d}=d$.
\item[($\beta$)] For any odd prime number $p$ with odd $v_p(\Delta)$ and
even $v_p(a)$, one has
$$\left(\frac{-a_p}{p}\right)=1\,;$$
\item[($\gamma$)] For any odd prime number $p$ with odd $v_p(\Delta)$ and
odd $v_p(a)$, one has
$$\left(\frac{-a_p\Delta_p}{p}\right)=1\,;$$
\item[($\delta$)] For any odd prime number $p$ with even $v_p(\Delta)$ and
odd $v_p(d)$, one has
$$\left(\frac{-\Delta_p}{p}\right)=1\,.$$
\end{itemize}
\item[(iv)] Eq.~\ref{eq-mordell1} is solvable for $m=4$ iff
$\Delta\in\sos (3)$.
\item[(v)] Eq.~\ref{eq-mordell1} is solvable for $m=5$ (and thus for
all $m\geq 5$).
\end{enumerate}
\end{theo}

Note that our formulation of the results in
the case $m=3$ is a somewhat streamlined version of the one
given by Mordell in his original article \cite{mo2}.

We now apply Mordell's results to the determination of $\Ss_m(n)$ for
$2\leq m\leq 7$.

\begin{theo}\label{m<8}
Let $n\in\NN$ and $T\in\ZZ$.
\begin{enumerate}
\item[(i)]  $T\in\Ss_2(n)$ if and only if
$2n-T^2\in \sos(1)$.
\item[(ii)]  $T\in\Ss_3(n)$ if and only if 
\begin{itemize}
\item either $T=\pm 3t$ and $n=3t^2$
for some $t\in\NN$, or 
\item $3n-T^2>0$ and the following holds:
write $3n-T^2=D_0D_1^2$ with $D_0,D_1\in \NN$ and $D_0$ squarefree, say,
$D_0=2^k3^\ell q_1\ldots q_r$ with $k,\ell\in\{ 0,1\}$,
$q_1,\ldots,q_r$ ($r\geq 0$) pairwise different primes with
$q_i\not\in\{ 2,3\}$.
Then 
\begin{enumerate}
\item[(a)] $q_i\equiv 1\bmod 6$ for all $1\leq i\leq r$, and
\item[(b)] if $\ell =1$ or $\gcd(3,T,n)=3$, then $k=1$.
\end{enumerate} 
\end{itemize}
\item[(iii)] {\em (Goldmakher, Pollack \cite{gp})} 
$T\in \Ss_4(n)$ if and only if $n\equiv T\bmod 2$ and
$4n-T^2\in\sos(3)$.  In particular, 
$\Ss_4'(n)=\{ T\in\TT_{m}(n)\,|\,4n-T^2\in\sos(3)\}$.
\item[(iv)] $T\in \Ss_m(n)$ for $5\leq m\leq 7$ if and only if
$n\equiv T\bmod 2$ and $mn-T^2\geq 0$.  In particular,
$\Ss_m(n)$ is full for $5\leq m\leq 7$.
\end{enumerate}
\end{theo}

\begin{proof}
(i) By Proposition \ref{equiv}, $T\in \Ss_2(n)$ if and only if
$[2,T,n]$ is a sum of two squares of integral linear forms.
Now $d=\gcd (2,T,n)\in\{ 1,2\}$, so $d\in \sos (2)$.
By Theorem \ref{thm-mordell}(i) and Proposition \ref{obvious},
it follows that  $T\in \Ss_2(n)$ if and only if $2n-T^2\in\sos (1)$.

(ii)  It is certainly necessary to have $3n-T^2\geq 0$ by Eq.~\ref{cond}.
By Proposition \ref{max}, $T\in \Ss_3(n)$ with $3n-T^2=0$ if and only if
$T=\pm 3t$ and $n=3t^2$ for some $t\in\NN$.  So assume from now on
$3n-T^2>0$. 
In view of Proposition \ref{equiv}, it suffices to apply the
conditions ($\alpha$)--($\delta$) in Theorem \ref{thm-mordell}
to the binary form $3X^2+2TXY+nY^2$, for which $a=3$,
$d=\gcd (3,T,n)=\gcd (3,2T,n)=\widetilde{d}\in\{ 1,3\}$ and
$\Delta=3n-T^2$.  Trivially, ($\alpha$) is satisfied.  As for
($\beta$) and ($\gamma$), the only odd primes $p$ with 
odd $v_p(\Delta)$ are the $q_i$'s, and $3$ if $\ell=1$.
Since $v_{q_i}(a)=v_{q_i}(3)=0$, condition ($\beta$) applies to these primes,
and with $a_{q_i}=3_{q_i}=3$ and using quadratic reciprocity,
it translates into
\begin{equation}\label{qi}
\left(\frac{-3}{q_i}\right)=\left(\frac{q_i}{3}\right)=1\,,
\end{equation}
or, equivalently, $q_i\equiv 1\bmod 3$ and therefore, since the
$q_i$ are odd, $q_i\equiv 1\bmod 6$.

If $\ell=0$, no further odd prime has to be considered in 
($\beta$) and ($\gamma$).
If $\ell=1$, then the odd prime $p=3$ has to be considered as well,
but then we are in the situation of ($\gamma$) since $v_3(a)=v_3(3)=1$.
Here, $a_3=1$, and the squarefree part of $\Delta_3$ is given
by $2^kq_1\ldots q_r$.  Hence, in view of Eq.~\ref{qi}, 
condition ($\gamma$) becomes
\begin{equation}\label{p=3}
\left(\frac{-2^kq_1\ldots q_r}{3}\right)=
\left(\frac{-2^k}{3}\right)=\left(\frac{2^{k+1}}{3}\right)=1\,,
\end{equation}
from which we derive the condition that if $\ell=1$, then $k=1$

As for condition ($\delta$), the odd primes $p$ with even $v_p(\Delta)$
are all odd primes $p\not\in\{ 3,q_1,\ldots, q_r\}$, and $p=3$ if
$\ell=0$, but we have odd $v_p(d)$ iff $d=p=3$.  Thus, condition
($\delta$) boils down to the condition that if $d=3$ and $\ell =0$,
then again $k=1$ (note that $-a_3\Delta_3=-\Delta_3$).

Thus, the previous two conditions can be summarized:
if $d=3$ or $\ell=1$, then $k=1$. 

(iii) and (iv) follow directly from Theorem \ref{thm-mordell}(iv),(v)
together with Proposition \ref{equiv}(ii).
\end{proof}

Let us illustrate the rather technical conditions in Theorem \ref{m<8}(ii)
with three examples.

\begin{example}
Let us determine $\Ss_3(n)$ for $n=42$. The $T\geq 0$ satisfying 
$3n-T^2=126-T^2\geq 0$ and $T\equiv n\bmod 2$ are
$T=0,2,4,6,8,10$.  Using the notations from 
Theorem \ref{m<8}(ii) 
and its proof and checking the criteria (a) and (b) there, 
we get the following table:
$$
\begin{array}{c|c|c|c|c|c|l}
T & \Delta=3n-T^2 & D_0 & k & \ell & d=\gcd(3,n,T) & T\in \Ss_3(n)?\\
\hline\hline
0 & 126 & 2\cdot 7 & 1 & 0 & 1 & \mbox{yes} \\\hline
2 & 122 & 2\cdot 61 & 1 & 0 & 1 &  \mbox{yes} \\\hline
4 & 110 & 2\cdot 5\cdot 11 &   1 & 0 & 1 & \mbox{no:}\ 5, 11\not\equiv 1\bmod 6  \\\hline
6 & 90 & 2\cdot 5 &   1 & 0 & 3 & \mbox{no:}\ 5\not\equiv 1\bmod 6 \\\hline
8 & 62 & 2\cdot 31 & 1 & 0 & 1 & \mbox{yes}  \\\hline
10 & 26 & 2\cdot 13 & 1 & 0 & 1 & \mbox{yes} 
\end{array}$$
Indeed, the only representation of $42$
as a sum of 
three squares of positive integers is $42=5^2+4^2+1^2$
(cf.~Theorem \ref{lehmer}(3)), from which we also easily get that 
$\Ss_3(42)=\{ \pm 10,\ \pm 8,\ \pm 2,\ 0\}$.
\end{example}

\begin{example}
We now determine $\Ss_3(n)$ for $n=43$.  
The $T\geq 0$ satisfying $3n-T^2=129-T^2\geq 0$ and $T\equiv n\bmod 2$ are
$T=1,3,5,7,9,11$.  Similarly to the previous example, we now get the following table:
$$
\begin{array}{c|c|c|c|c|c|l}
T & \Delta=3n-T^2 & D_0 & k & \ell & d=\gcd(3,n,T) & T\in \Ss_3(n)? \\
\hline\hline
1 & 128 & 2 & 1 & 0 & 1 & \mbox{yes} \\\hline
3 & 120 & 2\cdot 3\cdot 5 & 1 & 1 & 1 & \mbox{no:}\ 5\not\equiv 1\bmod 6  \\\hline
5 & 104 & 2\cdot 13 &   1 & 0 & 1 & \mbox{yes}  \\\hline
7 & 80 & 5 &   0 & 0 & 1 & \mbox{no:}\ 5\not\equiv 1\bmod 6  \\\hline
9 & 48 & 3 & 0 & 1 & 1 & \mbox{no:}\ (k,\ell)=(0,1)  \\\hline
11 & 8 & 2 & 1 & 0 & 1 & \mbox{yes}  
\end{array}$$
Indeed, the only representation of $43$
as a sum of 
three squares of positive integers is $43=5^2+3^2+3^2$
(cf.~Theorem \ref{lehmer}(3)), from which we also easily get that 
$\Ss_3(43)=\{ \pm 11,\ \pm 5,\ \pm 1\}$.
\end{example} 

\begin{example}
Let us finally determine $\Ss_3(n)$ for $n=75$.  
Note that here, $n=3t^2$ for $t=5$ from which it follows by 
Proposition~\ref{max}
that $T=\pm 3t=\pm 15\in \Ss_3(75)$.  Thus, the interesting cases are
the $T\geq 0$ satisfying $3n-T^2=225-T^2>0$ and $T\equiv n\bmod 2$, which are 
$T=1,3,5,7,9,11,13$.  Using the notations from 
Theorem \ref{m<8}(ii) and its proof, we now get the following table:
$$
\begin{array}{c|c|c|c|c|c|l}
T & \Delta=3n-T^2 & D_0 & k & \ell & d=\gcd(3,n,T) & T\in \Ss_3(n)? \\
\hline\hline
1 & 224 & 2\cdot 7 & 1 & 0 & 1 & \mbox{yes} \\\hline
3 & 216 & 2\cdot 3 & 1 & 1 & 3 & \mbox{yes}  \\\hline
5 & 200 & 2 &   1 & 0 & 1 & \mbox{yes}  \\\hline
7 & 176 & 11 &   0 & 0 & 1 & \mbox{no:}\ 11\not\equiv 1\bmod 6  \\\hline
9 & 144 & 1 & 0 & 0 & 3 &  \mbox{no:}\ (k,d)=(0,3) \\\hline
11 & 104 & 2\cdot 13 & 1 & 0 & 1 & \mbox{yes} \\\hline
13 & 56 & 2\cdot 7 & 1 & 0 & 1 & \mbox{yes}
\end{array}$$
Indeed, the only representations of $75$
as a sum of 
three squares of positive integers are $75=7^2+5^2+1^2=5^2+5^2+5^2$, from which we also easily
get that 
$\Ss_3(75)=\{ \pm 15,\ \pm 13,\ \pm 11,\ \pm 5,\ \pm 3,\ \pm 1\}$.
\end{example}

\section{Further results on $\Ss_m(n)$ for $m\geq 8$}

For $m\geq 5$, it turns out that $\Ss_m(n)$ only depends on the
maximal value in this set.  So we define
$$T_m^*(n)=\max\bigl(\Ss_m(n)\bigr)\,.$$

\begin{prop}\label{T-2}
Let $n,m\in\NN$,  $m\geq 5$.   If $0<T\in \Ss_m(n)$, then $T-2\in \Ss_m(n)$.
In particular,
$$\Ss_m(n) = \{ T\in\ZZ\,|\,T\equiv n\bmod 2,\ |T|\leq T_m^*(n)\}\,.$$
\end{prop}

\begin{proof} If $T=1\in \Ss_m(n)$, then by symmetry
$-1=-T=T-2\in \Ss_m(n)$.  Hence, it suffices to consider only the
case $T\geq 2$.  We use induction on $m$.  If $m=5$, then the fullness
of  $\Ss_m(n)$ (Proposition~\ref{m<8}) implies the result.
So suppose the result holds for a given $m\geq 5$, and let 
$2\leq T\in \Ss_{m+1}(n)$.  Then there exist $a_i\in\ZZ$, $0\leq i\leq m$,
with $\sum_{i=0}^ma_i=T\geq 2$ and $\sum_{i=0}^ma_i^2=n$. 

Suppose there exists some $i\in\{ 0,\ldots,m\}$ with $a_i\leq 0$,
say, $a_0\leq 0$.  Then $\sum_{i=1}^ma_i=T-a_0\geq T\geq 2$.
On the other hand, if each $a_i\geq 1$, then 
$\sum_{i=1}^ma_i\geq m> 4$.
Thus, we may assume that for $T'=\sum_{i=1}^ma_i$ and with 
$n'=\sum_{i=1}^ma_i^2$, we have 
$2\leq T'\in \Ss_m(n')$, and by induction hypothesis we have
$T'-2\in \Ss_m(n')$.  So there exist $b_i\in\ZZ$, $1\leq i\leq m$,
with $T'-2=\sum_{i=1}^mb_i$ and $n'=\sum_{i=1}^mb_i^2$, from which we get
$$T-2=a_0+T'-2=a_0+\sum_{i=1}^mb_i\quad\mbox{and}\quad
n=a_0^2+n'=a_0^2+\sum_{i=1}^mb_i^2\,,$$
which shows that $T-2\in  \Ss_{m+1}(n)$.
\end{proof}

\begin{coro}\label{T-2cor}
Let $n,m\in\NN$,  $m\geq 5$.   Then $\Ss_m(n)$ is full iff
$$T^*_m(n)=\left\{\begin{array}{lc}
\bigl[\sqrt{mn}\bigr] & \mbox{if 
$n\equiv \bigl[\sqrt{mn}\bigr]\bmod{2}$;}\\[1ex]
\bigl[\sqrt{mn}\bigr]-1 & \mbox{if 
$n\not\equiv \bigl[\sqrt{mn}\bigr]\bmod{2}$.}
\end{array}\right.$$
\end{coro}

\begin{prop}\label{bound}  Let $m,n\in\NN$.  
If $T\in\Ss_m'(n)$, then $T^2\leq m(n-1)+1$.
\end{prop}
\begin{proof}  We may assume $T\geq 0$.
Suppose $T\in\Ss_m'(n)$ with $T^2> m(n-1)+1$.  Since by assumption
$T^2<mn$, there exist $\ell,k,s\in\NN_0$ with $0\leq s\leq m-1$ and
$2\leq \ell\leq m-1$ (which forces $m\geq 3$) such that 
$$\begin{array}{rcl}
T & = & mk+s\,,\\
T^2 & = & m(n-1)+\ell\,.
\end{array}$$
Then 
$$T^2=(mk+s)^2=m^2k^2+2mks+s^2\equiv s^2\equiv\ell\bmod m\,,$$
and since $\ell\not\equiv 0,1\mod m$, we must have
$s\not\equiv 0,\pm 1\bmod m$, so we have $2\leq s\leq m-2$.
Note that this cannot happen for $m\leq 3$.  So let us from
now on assume $m\geq 4$.  We show by induction that for each $n$
we will get a contradiction.

If $n=1$, then $T^2> m(n-1)+1=1$ implies $T\geq 2$, but
$T\in\Ss_m'(1)=\{\pm 1\}$, a contradiction.   If $n=2$, then
$T^2> m(n-1)+1=m+1\geq 5$, hence $T\geq 3$, but
$T\in\Ss_m'(2)=\{0,\pm 2\}$, a contradiction.

So suppose $n\geq 3$.
By assumption, $T\in \Ss_m'(n)$, so we can find $a_i\in\ZZ$, $1\leq i\leq m$,
such that
\begin{equation}\label{induct}
\begin{array}{rcl}
T & = & \sum_{i=1}^m(k+a_i)=mk+\sum_{i=1}^ma_i=mk+s\,,\\[1ex]
n & = & \sum_{i=1}^m(k+a_i)^2=mk^2+2k\sum_{i=1}^ma_i+
\sum_{i=1}^ma_i^2\,.
\end{array}
\end{equation}
From this, we get $\sum_{i=1}^ma_i=s\geq 2$ and therefore
$n\geq n'=\sum_{i=1}^ma_i^2\geq 2$ as well.  In particular,
$s\in \Ss_m(n')$.  Furthermore, 
$$m^2k^2+2mks+s^2=T^2=mn-m+\ell = m^2k^2+2mks+m\sum_{i=1}^ma_i^2
-m+\ell$$
which implies
$$s^2=m\sum_{i=1}^ma_i^2 -m+\ell=m(n'-1)+\ell$$
and thus $s\in \Ss_m(n')$ with $m(n'-1)+1<s^2<mn'$.
Now if $n'<n$, this leads to a contradiction by induction (where
$T$ is replaced by $s$).
If $n'=n$, then Eq.~\ref{induct} implies that necessarily
$k=0$ and thus $2\leq T=s\leq m-2$.
If $n\geq m-1$, then $T^2\leq (m-2)(n-1)=m(n-1)-2(n-1)<m(n-1)+\ell$,
a contradiction.  If $n\leq m-2$, then since $T\in\Ss_m'(n)$,
we have by Proposition \ref{easy}(iii) that $T\leq n$, hence
$$n^2\geq T^2\geq m(n-1)+2\geq (n+2)(n-1)+2=n^2+n>n^2\,,$$
again a contradiction.

\end{proof}

\begin{coro}\label{max-coro}
Let $m,n\in\NN$ and $T^*=T^*_m(n)$.
\begin{enumerate}
\item[(i)] ${T^*}^2=mn$ if and only if there exists an $a\in\NN$ with
$n=ma^2$.  In this case, if $m\geq 5$, then $\Ss_m(n)=\Ss_m(ma^2)$ is
full.   In particular, there exist infinitely many values $n\in\NN$
for which $\Ss_m(n)$ is full. 
\item[(ii)] If there does not exist an $a\in\NN$ with $n=ma^2$, then
${T^*}^2\leq m(n-1)+1$.
\end{enumerate}
\end{coro}

\begin{remark}  Consider the situation $n=ma^2$ in (i) in the previous
corollary for $m\geq 5$, in which case $T^*_m(n)=T^*=ma$.  In this case,
the maximal value in $\Ss'_m(ma^2)$ is $ma-2$ by Proposition \ref{T-2}.
Now obviously 
$m(4a-1)\geq 3$ which implies 
that 
$$(ma-2)^2=m^2a^2-4ma+4\leq m^2a^2-m+1=m(n-1)+1$$
which of course would also follow from 
Proposition \ref{bound}.
\end{remark}

Having dealt with those $\Ss_m(n)$ where $n\leq m$ or $m\leq 7$ in 
Proposition \ref{easy} and Corollary \ref{m<8}, respectively,
we now focus on the case $n> m\geq 8$.  We need the
following technical lemma.

\begin{lemma}\label{goingdown}
Let $m,n,\ell,r\in\NN$, $s\in\NN_0$.  Suppose
that
\begin{equation}\label{goingdown-eq2}
\begin{array}{rcl}
m & > & \ell\,,\\
mn & > & r\,,\\
4\bigl((m-\ell)r-ms\bigr) & > & \ell m^2\,,
\end{array}
\end{equation}
and that
\begin{equation}\label{goingdown-eq1}
\{ T'\in\ZZ\,|\,T'\equiv n'\bmod 2,\ T'^2< (m-\ell)n'-s\}\subseteq
\Ss_{m-\ell}(n')
\end{equation}
for all $n'\in\NN$ of shape
$n'=n-\ell k^2$ for some $k\in\NN$.
Then
$$\{ T\in\ZZ\,|\,T\equiv n\bmod 2,\ T^2\leq mn-r\}\subseteq
\Ss_{m}'(n)\,.$$
\end{lemma}

\begin{proof}
Let $T\in \NN$ with $T\equiv n\bmod 2$ and $T^2\leq mn-r$.
If we can find $k\in\ZZ$ such that for $T'=T-\ell k$ and $n'=n-\ell k^2$
we have
$T'^2<(m-\ell)n'-s$ (which necessarily forces $n'>0$), then by assumption
and since $k\equiv k^2\bmod 2$, we have $T'\equiv n'\bmod 2$ and
thus $T'\in \Ss_{m-\ell}(n')$, i.e., there exist
$a_i\in\ZZ$, $1\leq i\leq m-\ell$, such that
$$T'=\sum_{i=1}^{m-\ell}a_i\quad\mbox{and}\quad n'= \sum_{i=1}^{m-\ell}a_i^2$$
and therefore 
$$T=\sum_{i=1}^{m-\ell}a_i + \ell\times k\quad\mbox{and}\quad 
n=\sum_{i=1}^{m-\ell}a_i^2+ \ell\times k^2\,,$$
which in turn implies that $T\in\Ss_m(n)$.

Now $T'^2<(m-\ell)n'-s$ translates into
$$(T-\ell k)^2<(m-\ell)(n-\ell k^2)-s$$
or, equivalently,
$$m\ell k^2-2T\ell k+T^2-n(m-\ell)+s<0\,.$$
We can find a 
$k\in\ZZ$ that satisfies this inequality iff the polynomial 
$P(X)=m\ell X^2-2T\ell X+T^2-n(m-\ell)+s\in\RR[X]$ has two roots 
$\rho_1<\rho_2$ in 
$\RR$ and there is an integer $k$ in the open interval $]\rho_1,\rho_2[$.
Now the roots are
$$\rho_{1,2}=\frac{T\ell\pm\sqrt{(m-\ell)\ell (mn-T^2)-m\ell s}}{m\ell}\,,$$
and they are real and distinct if and only if
$(m-\ell)(mn-T^2)-ms>0$, in which case
we have $]\rho_1,\rho_2[\,\cap\,\ZZ\neq\emptyset$ if 
$\rho_2-\rho_1>1$, i.e.,
$$2\sqrt{(m-\ell)\ell (mn-T^2)- m\ell s}>m\ell\,.$$
Hence, the existence of real roots $\rho_1<\rho_2$ with 
$]\rho_1,\rho_2[\,\cap\,\ZZ\neq\emptyset$ follows from
\begin{equation}\label{find-k}
4((m-\ell)(mn-T^2)-ms)>\ell m^2\,.
\end{equation}
Now by assumption $r\leq mn-T^2$.  Hence, Eq.~\ref{find-k}
is certainly satisfied if
$$4((m-\ell)r-ms)>\ell m^2\,,$$
but this holds by our assumptions in Eq.~\ref{goingdown-eq2}. 
\end{proof}

\begin{theo}\label{m=8-11}
\begin{enumerate}
\item[(i)] $\Ss_8(n)$ is full for all $n\in\NN$.
\item[(ii)] Let $n>m\in\{ 9,10,11\}$.  Then
$$\Ss_m'(n)=\{ T\in\ZZ\,|\,T^2\leq m(n-1)+1,\ T\equiv n\bmod 2\}\,.$$
\end{enumerate}
\end{theo}

\begin{proof}
We apply Lemma \ref{goingdown}.

(i) In the case $m=8$, we put $\ell=1$, $s=0$.  Note that if $T\in\TT_8(n)$
then $0<mn-T^2=8n-T^2$ and $n\equiv T\bmod 2$, from which
we conclude by working modulo $8$ that $mn-T^2\geq 4$.  So 
it suffices to show that we can
choose $r=4$ in the lemma in order to conclude fullness of $\Ss_8(n)$. 
Now $\Ss_7(n')$ is full for all $n'\in\NN$ by Theorem \ref{m<8},
therefore Eq.~\ref{goingdown-eq1} in the lemma is satisfied,
and so is Eq.~\ref{goingdown-eq2} because
$$112=4\bigl((m-\ell)r-ms\bigr)>\ell m^2=64\,.$$

(ii) In the case $m=9$, we choose $\ell=1$, and in the case
$m=10$ we choose  $\ell=2$.  In both cases we put $s=0$ and we
see that Eq.~ \ref{goingdown-eq1} in the lemma is satisfied
because of the fullness of $\Ss_8(n')$ that has been established in
part (i).  Since $m(n-1)+1=mn-(m-1)$, we put $r=m-1$.  We then have that
for $m=9$, $\ell=1$, $r=8$ and $s=0$, Eq.~\ref{goingdown-eq2}
is satisfied because
$$256= 4\bigl((m-\ell)r-ms\bigr)>\ell m^2=81\,,$$
and for $m=10$, $\ell=2$, $r=9$ and $s=0$, Eq.~\ref{goingdown-eq2}
is also satisfied because
$$288= 4\bigl((m-\ell)r-ms\bigr)>\ell m^2=200\,.$$
The lemma together with Proposition \ref{bound} then implies the result.

Let now $m=11$.  As before, in order to
determine $\Ss_{11}'(n)$ and because of
Proposition \ref{bound}, we only have to check for which $T\geq 0$ with
$T^2\leq m(n-1)+1=11n-10$ with $T\equiv n\bmod 2$, 
we have that $T\in \Ss_{11}'(n)$.  But then, for parity reasons,
we have either $T^2=11n-10$ or $T^2\leq 11n-12$.
Consider first the case $T^2=11n-10$. Then $T^2\equiv 1\bmod 11$ which
implies that $T\equiv\pm 1\bmod 11$, so this case can only
occur if there is a $k\in\NN$ with $T=11k\pm 1$ (note that $k\geq 1$ since
we assumed $n>m\geq 11$).  But then necessarily
$$T^2=11^2k^2\pm 2\cdot 11k+1=11n-10$$
and therefore $n=11k^2\pm 2k+1$, and we get the following representations
of $n$ by sums of $11$ squares:
$$\begin{array}{rcl}
n=11k^2+ 2k+1 = 10\times k^2 +(k+1)^2 & , & T=11k+1 =10\times k + (k+1)\,;\\
n=11k^2- 2k+1 = 10\times k^2 +(k-1)^2 & , & T=11k-1 =10\times k + (k-1)\,.
\end{array}$$
So indeed, if $T$ is such that $T^2=11n-10$, then $T\in\Ss_{11}'(n)$.

Finally, consider the case where $T^2\leq 11n-12$.  Here, we can
argue as in the cases $m=8,9,10$ but now with $r=12$, $\ell=3$, $s=0$
to conclude as before:
$$384=4\bigl((m-\ell)r-ms\bigr)>\ell m^2=363\,.$$
\end{proof} 

\begin{coro}\label{m=8-11-cor} 
Let $n\in\NN$, $n>m\in\{9,10,11\}$.
\begin{enumerate}
\item[(i)] $m=9$: $\Ss_9(n)$ is full iff $9n-2\notin\sos(1)$.
\item[(ii)] $m=10$:  $\Ss_{10}(n)$ is full iff either $n$ is odd and 
$10n-1,10n-5\notin\sos(1)$, or $n$ is even and $10n-4\notin\sos(1)$.
\item[(iii)] $m=11$:  $\Ss_{11}(n)$ is full iff 
$11n-2,11n-6,11n-8\notin\sos(1)$.
\end{enumerate}
\end{coro}

\begin{proof}  By Proposition \ref{bound} and Theorem \ref{m=8-11}, 
$\Ss_m(n)$ not being full 
is equivalent to the existence of
some $T\in\NN$ with $m(n-1)+1<T^2<mn$ and $T^2\equiv T\equiv n\bmod 2$.
For example, in (iii), this is equivalent to   the existence of 
some $T \in\NN$ with 
$11n-9\leq T^2\leq 11n-1$ and $T^2\equiv T\equiv n\equiv 11n\bmod 2$, and
since $T^2\equiv 0,1,3,4,5,9\bmod 11$, this is equivalent
to having some $T \in\NN$ with $T^2\in \{ 11n-2, 11n-6, 11n-8\}$.

(i) and (ii) can be shown by similar arguments, we leave the details to the
reader.
\end{proof}

\begin{example}\label{ex9}  Corollary \ref{m=8-11-cor}
states that if $n\geq 10$, then $\Ss_9(n)$  not being full
is equivalent to $9n-2$ being a square.
The smallest such $n$ is $19$:  $9\cdot 19-2=13^2$.
One could also easily check directly that $13\not\in \Ss_9(19)$.
But we also know by Theorem \ref{m=8-11}, that $11\in \Ss_9(19)$.
Indeed:
$$19=4\times 2^2 + 3\times 1^2=3^2+2^2+6\times 1^2\,.$$
We easily see that we get
all odd numbers $T$ between $-11$ and $11$ by suitably changing the signs 
of the coefficients
that are being squared in these representations.  Of course, this also
follows from Proposition~\ref{T-2}.
In particular, 
$$\Ss_9(19)=\{ T\in\ZZ\,|\,T\equiv 1\bmod 2,\ |T|\leq 11\}\,.$$ 
The above also shows that $\Ss_9(n)$ is full for all $9<n\leq 18$.
\end{example}

\begin{example} Similarly as before, we find that
if $n\geq 11$, then $\Ss_{10}(n)$  not being full is
equivalent to the existence of  
some $T\in\NN$ with $T\equiv n\bmod 2$ and $T^2\in \{10n-1,10n-4,10n-5\}$.
The smallest such $n$ is $17$: $10\cdot 17-1=13^2$.
Similarly as in the previous example, we know that $11\in \Ss_{10}(17)$, indeed:
$$17=3\times 2^2 + 5\times 1^2=3^2+8\times 1^2$$
and hence
$$\Ss_{10}(17)=\{ T\in\ZZ\,|\,T\equiv 1\bmod 2,\ |T|\leq 11\}\,.$$ 

The smallest such $n$ with $n$ even is $n=20$:  $10\cdot 20-4=14^2$.
We have that $12\in \Ss_{10}(20)$:
$$20=4\times 2^2 + 4\times 1^2 = 3^2 + 2^2 + 7\times 1^2$$
and hence
$$\Ss_{10}(20)=\{ T\in\ZZ\,|\,T\equiv 0\bmod 2,\ |T|\leq 12\}\,.$$ 
\end{example}

\begin{example}  Similarly as before, we find that the smallest
$n>11$ for which $\Ss_{11}(n)$ is not full is given by $n=18$:
$14^2=11\cdot 18-2$.
But we know that $12\in \Ss_{11}(18)$, indeed:
$18=3\times 2^2+6\times 1^2=3^2+9\times 1^2$.  Hence,
$$\Ss_{11}(18)=\{ T\in\ZZ\,|\,T\equiv 0\bmod 2,\ |T|\leq 12\}\,.$$ 
\end{example}

If $n\leq m$, then fullness of $\Ss_m(n)$ is dealt with in
Corollary \ref{small-n}.  In view of the above examples, 
if $n>m$ we still can expect fullness provided $n$ is `close' to $m$. 
The following corollary also explains the above examples in more
generality.

\begin{coro}\label{m>9} 
Let $n,m\in\NN$.
\begin{enumerate}
\item[(i)] If $10\leq n\leq 18$, then $\Ss_9(n)$ is full.
$\Ss_9(19)$ is not full.
\item[(ii)]  Let $m\geq 10$.  If $m< n\leq m+6$, 
then $\Ss_m(n)$ is full.   $\Ss_m(m+7)$ is not full.
\end{enumerate}
\end{coro}

\begin{proof} (i) This follows from the arguments in Example \ref{ex9}.

(ii)  Write $T^*=T^*_m(n)$, $m< n\leq m+7$.  Under the assumptions, 
$n$ is not of shape $ma^2$ for some
$a\in\NN$.  Thus, for parity reasons, in order to have fullness, 
it is necessary and sufficient that if $c\in\NN$
with $c=n\bmod 2$,  $c^2<mn$, $(c+2)^2\geq mn$, then
$c\in\Ss_m(n)$, in which case $T^*=c$.

Let $n=m+k$.  For $1\leq k\leq 3$, the only $c$ satisfying
these conditions is
$c=n-2=m+k-2$.  But then we can write $n$ as a sum of $m-3+k\leq m$
squares as follows:
$$n=(m-4+k)\times 1^2 +2^2\quad\mbox{with}\quad 
c=m+k-2=(m-4+k)\times 1 + 2\,.$$
For $4\leq k\leq 6$,  the only $c$ satisfying these conditions is 
$c=n-4=m+k-4$.  But then we can write $n$ as a sum of $m-6+k\leq m$
squares as follows:
$$n=(m-8+k)\times 1^2+2\times 2^2
\quad\mbox{with}\quad c=m+k-4=(m-8+k)\times 1 + 2\times 2\,.$$
This shows that in all these cases, we have indeed $c\in\Ss_m(n)$,
implying fullness.

Now if $n=m+7$, then $c=m+3$ satisfies the above conditions.
But then $m(n-1)+1=m^2+6m+1<c^2=(m+3)^2$, so $c\not\in \Ss_m(m+7)$
by Proposition \ref{bound}, hence $\Ss_m(m+7)$ is not full.
\end{proof}

\begin{example}  For each $m\geq 9$, there exist infinitely many
$n> m$ for 
which $\Ss_m(n)$ is not full. To show this, we just have to find
$n$ and $T$ with $T\equiv n\bmod 2$ and $m(n-1)+1 < T^2 <mn$.

If $m\geq 10$, let $r\in\NN$ and put $T=2mr+3$ and $n=4mr^2+12r+1$.
Both $n$ and $T$ are odd, and we have
$$m(n-1)+1<T^2=4m^2r^2+12mr+9=m(n-1)+9<mn=m(n-1)+m\,.$$
Then $T\in\TT_m(n)$, but by Proposition~\ref{bound},
$T\not\in\Ss_m'(n)$.

If $m=9$, let $r\in\NN$ and put $T=18r-5$ and $n=36r^2-20r+3$. 
Both $n$ and $T$ are odd, and we get
$$9(n-1)+1=T^2-6<T^2<T^2+2=9n\,.$$
Again, $T\in\TT_m(n)$ but by Proposition~\ref{bound},
$T\not\in\Ss_m'(n)$.  Note that for $r=1$, we recover
the case $n=19$ and $T=13$ from Example \ref{ex9}. 
\end{example}

\begin{example} Let $n>m=12$ and let
$T\in\NN_0$ with $T\equiv n\bmod 2$.  Proposition \ref{bound} shows that if
$T^2<mn$, then a necessary condition for  
$T\in\Ss_{12}(n)$ is that $T^2\leq 12(n-1)+1$. However, this is in general
not sufficient.  Take $T=17$ and $n=25$.  Then $17^2=289=12\cdot 24+1$.
One easily checks that $17\not\in\Ss_{12}(25)$.  Indeed,
$\Ss_{12}(25)=\{ T\in\ZZ\,|\,T\equiv 1\bmod 2,\ |T|\leq 15\}$.
The value $15$ can be obtained from the representations
$$25= 5\times 2^2+5\times 1^2=3^2+2\times 2^2+8\times 1^2\,.$$
\end{example}

\begin{problem} If $m\geq 5$, then knowing $T^*_m(n)$ yields a
full description of $\Ss_m(n)$ by Proposition \ref{T-2}.  Thus,
we obtain the following rather natural problem:
Find an explicit description or formula for $T^*_m(n)$ in terms
of (properties of) $m$ and $n$.
\end{problem}

\section{Variations of the problem}
In this section we consider a variation of the original problem
concerning sums of four squares.  This problem has been studied by
by Z.-W.\ Sun et.\ al.\ in a series of papers
\cite{sun1}, \cite{sun2}, \cite{sun3}, \cite{sun4}.  The purpose
of this section is to show how our methods, in particular Mordell's
results on sums of squares of linear forms, can be used
to recover and extend some of the results by 
Z.-W.\ Sun et.\ al..
First, we generalize some of the problems they consider 
to $m$ squares for $m\in\NN$.  
Let $(a_1,\ldots,a_m)\in\ZZ^m\setminus\{ (0,\ldots,0)\}$, and $n\in\NN$.
This time, we ask for which $T\in\ZZ$ the following system of
diophantine equations has a solution $(x_1,\ldots,x_m)\in\ZZ^m$:
\begin{equation}\label{eq-sun}
\begin{array}{rcrcrcrcl}
a_1x_1 & + & a_2x_2 & + & \ldots & + & a_mx_m & = & T\\[1ex] 
x_1^2 & + & x_2^2 & + & \ldots & + & x_m^2 & = & n
\end{array}
\end{equation}
Note that the solvability of Eq.~\ref{eq-sun} with $x_i\in\ZZ$ 
is invariant under
sign changes of the $a_i$ (just change the signs of the 
corresponding $x_i$) and permutation of the indices, so that
we may assume from now on that
$$a_1\geq a_2\geq \ldots\geq a_m\geq 0,\ a_1\geq 1\,.$$ 
We define
$$A(x_1,x_2,\ldots,x_m)=\sum_{i=1}^ma_ix_i\in
\ZZ[x_1,x_2,\ldots,x_m]\quad\mbox{and}\quad a=\sum_{i=1}^ma_i^2\,.$$
In  analogy to before,  we define the set
$$\Ss_{m,A}(n)=\{ T\in\ZZ\,|\,\mbox{Eq.~\ref{eq-sun} has a solution
$\xv\in\ZZ^m$}\}\,.$$
\begin{prop}\label{sma-cs}
Let $n,m,a,A$ be as above, and let $T\in\ZZ$.  Let
$d=\gcd(a_1,\ldots,a_m)\in\NN$ and let
$a_i=da_i'$.  Put $a'=\sum_{i=1}^ma_i'^2$, so in particular
$a=d^2a'$.
\begin{enumerate}
\item[(i)]
If $T\in\Ss_{m,A}(n)$, then $T^2\leq an$.
\item[(ii)] $T\in\Ss_{m,A}(n)$ with  $T^2=an$ if and only if
there exists $b\in\NN$ with $n=a'b^2$, in which case 
$T\in\{\pm a'bd\}$.
\end{enumerate}
\end{prop}
\begin{proof}  (i) follows from the Cauchy-Schwarz
inequality as in Eq.~\ref{cs} with $\ein$ replaced by
$\av=(a_1,\ldots,a_m)$ and $m$ replaced by $a$.

(ii) Put $\av'=(a_1',\ldots,a_m')$ so that $\av=d(a_1',\ldots,a_m')$.
If $n=a'b^2$, then $an=a'^2b^2d^2=(a'bd)^2$.   
For $\xv=\pm (ba_1',\ldots,ba_m')$ we then have 
$\xv\cdot\xv=n$ and  $A(\xv)=\av\cdot\xv=\pm a'bd\in \Ss_{m,A}(n)$.

Conversely, if $T\in\Ss_{m,A}(n)$ with  $T^2=an$, then
there exists $\xv\in\ZZ^m$ with $n=\xv\cdot\xv$ and 
$an=(\av\cdot\xv)^2$.  Since $\av\cdot\av =a$, we have
equality in the Cauchy-Schwarz inequality which implies that
$\xv$ depends linearly on $\av$ and hence on
$\av'$. Since $\xv,\av'$ have coefficients in $\ZZ$ and because
of $\gcd(a_1',\ldots,a_m')=1$, there exists $b\in\ZZ$ such that
$\xv=b\av'$, hence $n=\xv\cdot\xv=a'b^2$ and 
$T=A(\xv)=\av\cdot\xv=a'bd$.
\end{proof}
Similarly to what we did in Section 2, we define 
$$\Ss_{m,A}'(n)=\{ T\in\Ss_{m,A}(n)\,|\,T^2< an\}\,.$$
Note that
$T\in\Ss_{m,A}(n)$ if and only if
there exists $(x_1,\ldots,x_m)\in\ZZ^m$ such that
$$aX^2+2TXY+nY^2=\sum_{i=1}^m(a_iX+x_iY)^2\in\ZZ[X,Y]\,.$$
Thus, a necessary condition for $T\in\Ss_{m,A}(n)$ is that 
one can find $\alpha_i,\beta_i\in\ZZ$ so that the following 
equation holds:
\begin{equation}\label{eq-a}
aX^2+2TXY+nY^2=\sum_{i=1}^m(\alpha_iX+\beta_iY)^2\in\ZZ[X,Y]\,.
\end{equation}
Proposition \ref{obvious} and Theorem \ref{thm-mordell}
tell us exactly when this necessary condition is satisfied.
However, this condition is generally not sufficient.
But there are cases, where sufficiency also holds.
Note first that in Eq.~\ref{eq-a}, 
we may again assume (after permuting the summands
and changing signs if necessary) that $\alpha_1\geq\alpha_2
\geq\ldots\geq\alpha_m\geq 0$.   If $a=\sum_{i=1}^ma_i^2$
is essentially the only decomposition of $a$ into a sum of
$m$ squares, i.e., if for any other decomposition
$a=\sum_{i=1}^m\alpha_i^2$ with $\alpha_i\in\ZZ$ and
$\alpha_1\geq\alpha_2\geq\ldots\geq\alpha_m\geq 0$, we have 
$a_i=\alpha_i$, $1\leq i\leq m$, then the solvability of 
Eq.~\ref{eq-a} will be equivalent to the solvability of 
Eq.~\ref{eq-sun}.

This leads to the definition of the \emph{partition number}
$P_m(n)$ for $n,m\in\NN$:
$$P_m(n)=\bigl|\{ (a_1,\ldots,a_m)\in\NN_0^m\,|\,
a_1\geq a_2\geq\ldots\geq a_m,\ {\textstyle \sum_{i=1}^ma_i^2=n}\}\bigr|\,,$$
the number of essentially different partitions of $n$ into $m$ integer
squares. 
D.H.~Leh\-mer \cite{le} studied the question for which $n$ one has 
$P_m(n)=1$.  He
gave a full solution for $m\neq 3$, and he provided a conjecture for
$m=3$.  This conjecture has been later confirmed (albeit with a correction)
by Bateman and Grosswald \cite{bg}.  Their proof uses the classification
of discriminants of binary quadratic forms of class number $\leq 4$
which had been itself a conjecture at the time and which only later on
was fully established by Arno \cite{ar}.  Here is the complete result.
\begin{theo}\label{lehmer}
Let $n,m\in\NN$. 
\begin{enumerate}
\item[(1)] $P_1(n)=1$ iff $n\in\sos(1)$.
\item[(2)] $P_2(n)=1$ iff $n=2^kq^2,\ 2^kq^2p$ where $k\in\NN_0$,
$q\in\NN$ is an odd integer having only prime factors $\equiv 3\bmod 4$,
and $p$ is a prime with $p\equiv 1\bmod 4$.
\item[(3)] $P_3(n)=1$ iff $n=4^kc$ with
$k\in\NN_0$ and $c=1$, $2,\ 3,\ 5,\ 6$, $10,\ 11$, $13$, $14$, $19,\ 21$, 
$22,\ 30,\ 35,\  37$, $42$, $43,\ 46,\ 58,\ 67$, $70,\ 78$, $91$, 
$93,\ 115,\ 133$, $142$, $163$, $190,\ 235,\ 253,\ 403,\ 427$.
\item[(4)] $P_4(n)=1$ iff $n=1,\ 3,\ 5,\ 7,\ 11,\ 15,\ 23,\ 2\cdot4^k,
\ 6\cdot4^k,\ 14\cdot4^k$ with $k\in\NN_0$.
\item[(5)] $P_5(n)=1$ iff $n=1,\ 2,\ 3,\ 6,\ 7,\ 15$.
\item[(6)] Let $m\geq 6$.  Then $P_m(n)=1$ iff $n=1,\ 2,\ 3,\ 7$.
\end{enumerate}
\end{theo}
Back to our original problem.  Let
\begin{equation}\label{setup}
\begin{array}{cl}
\bullet & a_i\in\NN_0,\ 1\leq i\leq m,\ \mbox{with}\ 
0\neq a_1\geq a_2\geq\ldots\geq a_m,\\[1ex]
\bullet & A(x_1,\ldots,x_m)=\sum_{i=1}^ma_ix_i, \\[1ex]
\bullet & a=\sum_{i=1}^ma_i^2.
\end{array}
\end{equation}
Together with the preceding remarks, we now have established the
following. 
\begin{prop}\label{sma}
Let $n,m\in\NN$ and let $a_i,\ a,\ A$ be as in Eq.~\ref{setup}.
If $T\in\Ss_{m,A}(n)$ then there exist $\alpha_i,\beta_i\in\ZZ$, 
$1\leq i\leq m$ that satisfy
Eq.~\ref{eq-a}.  The converse holds if in addition $P_m(a)=1$.
\end{prop}

In \cite{sun1}, \cite{sun2}, \cite{sun3}, \cite{sun4}, the authors
study among other things the sets $\Ss_{4,A}(n)$ for certain linear polynomials
$A\in\ZZ[x_1,x_2,x_3,x_4]$ as above.  They give (partial) results on the existence
of certain types of elements contained in $\Ss_{4,A}(n)$.
The following corollary allows to easily
recover and to extend their results in that context.

\begin{coro}\label{coro-sun} Let $n\in\NN$ and let 
$a_i,\ a,\ A$ be as in Eq.~\ref{setup}
with $m=4$.  Assume $P_4(a)=1$.  Then 
$$\Ss_{4,A}'(n)=\{ T\in\ZZ\,|\,0<an-T^2\in \sos(3)\}\,.$$
If $\gcd (a_1,a_2,a_3,a_4)=d\in\NN$ and there exists $b\in\NN$
with $n=ab^2d^{-2}$, then
$\Ss_{4,A}(n)=\Ss_{4,A}'(n)\cup\{\pm abd^{-1}\}$.  Otherwise,
$\Ss_{4,A}(n)=\Ss_{4,A}'(n)$.
\end{coro}
\begin{proof}
This follows readily from Propositions \ref{sma-cs}, \ref{sma}
together with Theorem~\ref{thm-mordell}(iii).
\end{proof}

\begin{example}
Sun et.~al. studied polynomials of type $A(x_1,x_2,x_3,x_4)=\sum_{i=1}^4a_ix_i$
with $P_4\bigl(\sum_{i=1}^4a_i^2\bigr)=1$, for
example $(a_1,a_2,a_3,a_4)=(1,1,0,0)$ (\cite[Th.~1.2(i)]{sun1}, 
\cite[Th.~1.1(iii)]{sun3}), $(2,1,0,0)$ (\cite[Th.~1.2(ii)]{sun1}),
$(3,2,1,1)$ (\cite[Th.~1.5]{sun2}), $(3,1,1,0)$ (\cite[Th.~1.6(ii)]{sun2}), $(2,1,1,1)$ (\cite[Th.~1.7(ii)]{sun2}), $(3,2,1,0)$  (\cite[Th.~1.7(iv)]{sun2}),
$(2,1,1,0)$ (\cite[Th.~1.4(i)]{sun3}).  They prove various results of the type
that $\Ss_{4,A}(n)$ contains a square, or twice a square, or a cube, or
twice a cube, or a power of $4$, or a power of $8$, or the like.  We refrain
from presenting their results in detail.  Suffice it to say that
all these types of results for the above mentioned polynomials
can now readily be checked or recovered 
using our explicit and complete description of  $\Ss'_{4,A}(n)$ in all 
these cases by Corollary~\ref{coro-sun}.

Just as an illustration, let us consider the case $(3,1,1,0)$.  We will show that
$\Ss_{4,A}(n)$ always contains a power of $4$. By Corollary \ref{coro-sun},
it suffices to show that $11n-(4^{m})^2\in\sos (3)$ for some $m\in\NN_0$
to conclude that  $4^m\in \Ss_{4,A}(n)$.

Below is a list of such elements $T=4^m$, where we write
$11n=4^{2t+r}(8k+s)$ with $t,k\in\NN_0$, $r\in\{ 0,1\}$, 
$s\in\{ 1,2,3,5,6,7\}$.  Note that then $11$ must divide $8k+s$,
so $8k+s\geq 11$, hence $8k+s-2^\ell>0$ for $\ell\leq 3$, and
if $s\neq 3$ then $8k+s\geq 22$ in which case $8k+s-2^\ell>0$ for $\ell\leq 4$.
$$\begin{array}{ll|l} 
11n & & T\\\hline\hline
4^{2t}(8k+s)\rule{0ex}{2.3ex} & s\neq 1,5 & 4^{t}\\
 & s\neq 3,7 & 4^{t+1}\\\hline
4^{2t+1}(8k+s)\rule{0ex}{2.3ex} & s\neq 2,6 & 4^{t}\\
 & s\neq 3 & 4^{t+1}
\end{array}$$

For example, if $11n=4^{2t+1}(8k+s)$ with $s\neq 2,6$, then we have 
$11n-(4^t)^2=4^{2t}(32k+4s-1)$ with $32k+4s-1\equiv 3\bmod 8$, showing that
indeed $11n-(4^t)^2\in\sos (3)$.
\end{example}

\subsection*{Acknowledgment}  I am grateful to Professor Zhi-Wei Sun for 
pointing out an error in an earlier version of this article.

\end{document}